\documentclass[10pt,reqno]{amsart}
\usepackage{latexsym,amsmath,amssymb}

\begin{document}

\newcommand{\bfi}{\bfseries\itshape}

\newtheorem{theorem}{Theorem}
\newtheorem{acknowledgment}[theorem]{Acknowledgment}
\newtheorem{corollary}[theorem]{Corollary}
\newtheorem{definition}[theorem]{Definition}
\newtheorem{example}[theorem]{Example}
\newtheorem{lemma}[theorem]{Lemma}
\newtheorem{notation}[theorem]{Notation}
\newtheorem{proposition}[theorem]{Proposition}
\newtheorem{remark}[theorem]{Remark}

\numberwithin{theorem}{section}
\numberwithin{equation}{section}

\newcommand{\ad}{{\rm ad}}
\newcommand{\Der}{{\rm Der}\,}
\newcommand{\End}{{\rm End}\,}
\newcommand{\id}{{\rm id}}
\newcommand{\GL}{{\rm GL}}
\newcommand{\Ker}{{\rm Ker}\,}
\newcommand{\Ran}{{\rm Ran}\,}
\newcommand{\redrtimes}{\,\widetilde{\rtimes}\,}
\newcommand{\redtimes}{\,\widetilde{\times}\,}
\newcommand{\spa}{{\rm span}\,}

\newcommand{\CC}{{\mathbb C}}
\newcommand{\RR}{{\mathbb R}}

\newcommand{\Ac}{{\mathcal A}}
\newcommand{\Bc}{{\mathcal B}}
\newcommand{\Cc}{{\mathcal C}}
\newcommand{\Hc}{{\mathcal H}}
\newcommand{\Vc}{{\mathcal V}}
\newcommand{\Yc}{{\mathcal Y}}
\newcommand{\Zc}{{\mathcal Z}}

\newcommand{\Xg}{{\mathfrak X}}
\newcommand{\Yg}{{\mathfrak Y}}

\newcommand{\ag}{{\mathfrak a}}
\newcommand{\bg}{{\mathfrak b}}
\newcommand{\cg}{{\mathfrak c}}
\newcommand{\dg}{{\mathfrak d}}
\renewcommand{\gg}{{\mathfrak g}}
\newcommand{\hg}{{\mathfrak h}}
\newcommand{\kg}{{\mathfrak k}}
\newcommand{\zg}{{\mathfrak z}}

\newcommand{\ZZ}{\mathbb Z}
\newcommand{\NN}{\mathbb N}
\newcommand{\KK}{\mathbb K}

\makeatletter
\title[On Kirillov's lemma]{On Kirillov's lemma for nilpotent Lie algebras}
\author{Ingrid Belti\c t\u a and Daniel Belti\c t\u a}
\address{Institute of Mathematics ``Simion Stoilow'' 
of the Romanian Academy, 
P.O. Box 1-764, Bucharest, Romania}
\email{ingrid.beltita@gmail.com, Ingrid.Beltita@imar.ro}
\email{beltita@gmail.com,  Daniel.Beltita@imar.ro}
\keywords{nilpotent Lie algebra; coadjoint orbit}
\subjclass[2000]{Primary 17B30; Secondary 22E25, 22E27}
\date{15 January 2015}
\makeatother

\begin{abstract} 
We establish a sharpening of Kirillov's lemma on nilpotent Lie algebras with 1-dimensional center 
and use it to study the structure of 3-step nilpotent Lie algebras. 
\end{abstract}

\maketitle


\section{Introduction} 

Kirillov's famous lemma (\cite[Lemma 4.1]{Ki62}) says that 
for every real nilpotent Lie algebra $\gg$ with the 1-dimensional center~$\zg$ 
there exist an ideal $\gg_0$ and the elements $X\in\gg$, $Y\in \gg_0$ and $0\ne Z\in\zg$ 
satisfying the the Heisenberg commutation relation $[X,Y]=Z$ and in addition $\gg=\gg_0\dotplus\RR X$ 
and $\gg_0=\{V\in\gg\mid[V,Y]=0\}$. 
This observation has many deep implications in representation theory and Lie theory; 
see for instance \cite{Ki62} and \cite{CG90}. 
Let us briefly recall one of the basic reasons for studying nilpotent Lie algebras with 1-dimensional center: 
If $G$ is any connected, simply connected, nilpotent Lie group, 
then for every unitary irreducible representation $\pi\colon G\to\Bc(\Hc)$ 
there exists a closed connected subgroup $N$ of the center of $G$ such that $G/N$ has 1-dimensional center 
and $N\subseteq\Ker\pi$. 
Hence there exists a unique representation $\pi_0\colon G/N\to\Bc(\Hc)$ with $\pi(x)=\pi_0(xN)$ for every $x\in G$. 
Thus, in the study of single unitary representations of nilpotent Lie groups, one may always assume that the center of the Lie group under consideration is 1-dimensional. 

In the present paper we record a sharpening of the above result (Theorem~\ref{K0}), 
and the main results of this paper (Theorems \ref{new1} and \ref{struct}) 
show its impact on certain aspects of 
the structure theory of nilpotent Lie algebras with low nilpotency index.  
The circle of ideas approached here goes back to some methods developed in \cite{Lu83} and \cite{Ra85}, 
which are also very important for our investigation.  
We will relate this to some classification problems that have attracted a notable interest, 
witnessed by the numerous conjectures and problems left open in this area; 
see for instance 
\cite{GH93}, \cite{GK00}, \cite{DT00}, \cite{GKM04}, \cite{Bu06}, \cite{PT09}, 
\cite{CGS12}, \cite{HT13}, \cite{Mi13}, 
and the references therein. 

In Section~\ref{Sect2} we introduce some terminology that is useful for describing various classes 
of Lie algebras which occur in the paper. 
The key technical result of our paper (Theorem~\ref{K0}) 
is the aforementioned sharp version of Kirillov's lemma 
and is established in Section~\ref{Sect3} along with some auxiliary properties of the second center of a Lie algebra. 
Sections \ref{Sect4}~and~\ref{Sect5} present some applications of this theorem 
to the structure theory of 3-step nilpotent Lie algebras with 1-dimensional center, 
which in particular sheds some fresh light on the results of \cite{Ra85}. 
Also, since the present investigation has been prompted by our research on the Weyl calculus 
for square-integrable unitary representations of nilpotent Lie groups 
(see \cite{Pe94}, \cite{BB11}, \cite{BB12}), 
we pay a special attention to the characterization of the Lie algebras of this type (Theorem~\ref{flat}). 
Finally, Section~\ref{Sect6} includes the collection of examples of nilpotent Lie algebras of dimension~$\le 6$ 
over~$\RR$ 
to which our main results are applicable and which illustrate the various situations that can occur. 

\section{Terminology}\label{Sect2}

\begin{notation}
\normalfont
We denote by $\KK$ an arbitrary field of characteristic different from~2. 
Every vector space in the present paper is implicitly assumed to be a finite-dim\-en\-sional vector space over~$\KK$. 
In particular, this applies for the Lie algebras we will be working with. 

In any vector space we will denote the subset $\{0\}$ simply by $0$. 
The symbol $\dotplus$ stands for the direct sum of vector spaces, 
and we will use the notation $\ag\unlhd\gg$ for the fact that $\ag$ is an ideal of the Lie algebra~$\gg$. 
\end{notation}

\begin{notation}
\normalfont 
If $\gg$ is a Lie algebra, then its \emph{center} is  
$$\Zc(\gg):=\Cc_1\gg:=\{X\in\gg\mid[\gg,X]=0\}$$ 
while its \emph{second center} is 
$$\Cc_2\gg:=\{X\in\gg\mid[\gg,[\gg,X]]=0\}.$$
For every subset $\ag\subseteq\gg$ we also define 
$$\Cc(\ag:\gg):=\{X\in\gg\mid[\ag,X]=0\}$$
which is called the \emph{centralizer} of $\ag$ in $\gg$. 
\end{notation}

\begin{remark}
\normalfont 
For every Lie algebra $\gg$ we have $\Zc(\gg)\subseteq\Cc_2\gg$ 
and both these subsets are characteristic ideals of $\gg$ (see \cite{Bo71}). 
Moreover, $\Zc(\gg)$ is an abelian Lie algebra, while $\Cc_2\gg$ is a 2-step nilpotent Lie algebra, 
as noted in Proposition~\ref{L0.9} below. 
\end{remark}

\begin{remark}\label{L0}
\normalfont 
For later use, we recall that every 2-step nilpotent Lie algebra with 1-dim\-en\-sional center is a Heisenberg algebra. 
\qed
\end{remark}

The nilpotent Lie algebras of type~($A$) introduced below will play an important role in this paper.  

\begin{definition}\label{def_A}
\normalfont
Let $\kg$ be a nilpotent Lie algebra with $\dim\Zc(\kg)=1$. 
\begin{itemize}
\item The algebra $\kg$ is said to be \emph{of type ($A$)} if either $[\Cc_2\kg,\Cc_2\kg]=0$ or $\Cc_2\kg=\kg$. 
\item The algebra $\kg$ is said to be \emph{of type ($A_{+}$)} if either $\Cc_2\kg$ is a maximal abelian subalgebra of $\kg$ 
or $\Cc_2\kg=\kg$. 
\item A \emph{grading of type ($A_{+}$)} of $\kg$ is a triple $\gamma=(\zg,\cg,\Vc)$ such that we have the direct sum decomposition 
$\kg=\zg\dotplus\cg\dotplus\Vc$ with $\zg=\Zc(\kg)$, $[\cg,\cg]=0$, $[\Vc,\Vc]\subseteq\cg$, 
and moreover the Lie bracket of $\kg$ gives rise to a duality pairing 
$[\cdot,\cdot]\colon\cg\times\Vc\to\Zc(\kg)\simeq\KK $. 
\end{itemize}
\end{definition}

Every Lie algebra with a grading of type ($A_{+}$) is clearly 3-step nilpotent. 
The full connection between the 3-step nilpotent Lie algebras of type ($A_{+}$) and gradings of type ($A_{+}$) 
will be established in Theorem~\ref{new1} below. 
Some 3-step nilpotent Lie algebras of type~($A$) in the sense of the above definition 
were also discussed in connection with the sub-Riemannian geometry in \cite[Sect. 2]{MA11}. 

Not every nilpotent Lie algebra with 1-dimensional center is of type~($A$); 
see Example~\ref{N6N1} for a specific example in this connection. 
We will prove in Theorem~\ref{struct}\eqref{struct_item1} that in some sense
the study of nilpotent Lie algebras with 1-dimensional center reduces to the study of Lie algebras of type~($A$) 
and of the Heisenberg algebras.

\begin{remark}
\normalfont
For every Lie algebra $\kg$ we have: 
\begin{itemize}
\item $\Cc_2\kg=\kg$ in the above definition means that $\kg$ is a Heisenberg algebra;  
\item $[\Cc_2\kg,\Cc_2\kg]=0$ $\iff$ $\Cc_2\kg\subseteq\Cc(\Cc_2\kg:\kg)$; 
\item $\Cc_2\kg$ is a maximal abelian subalgebra $\iff$ $\Cc_2\kg=\Cc(\Cc_2\kg:\kg)$.  
\end{itemize}
A reference for the latter assertion is \cite[\S 4, Exerc. 6]{Bo71}.
\end{remark}

\begin{remark}
\normalfont
If $\kg$ is a 3-step nilpotent Lie algebra with $\dim\Zc(\kg)=1$ and $\Cc_2\kg\ne\kg$, 
then $[\kg,[\kg,\kg]]=\Zc(\kg)\ne0$. 

In fact, since $[\kg,[\kg,[\kg,\kg]]]=0$, it follows that $[\kg,[\kg,\kg]]\subseteq\Zc(\kg)$. 
Since $\dim\Zc(\kg)=1$, we see that we could have $[\kg,[\kg,\kg]]\ne\Zc(\kg)$ if and only if $[\kg,[\kg,\kg]]=0$.
However, if $[\kg,[\kg,\kg]]=0$, then $\Cc_2\kg=\kg$, which is a contradiction with the hypothesis. 
\end{remark}

\begin{definition}\label{compatgr}
\normalfont 
Let $\gg$ be a Lie algebra and $\kg$ and $\ag$ be two subalgebras of~$\gg$ such that $[\ag,\kg]\subseteq\kg$. 
Also assume that $\kg$ has a grading $\gamma=(\zg,\cg,\Vc)$ of type ($A_{+}$). 
We say that $\ag$ is \emph{compatible with the grading $\gamma$} if 
$[\ag,\Vc]\subseteq\cg$ and $[\ag,\zg+\cg]=0$. 
\end{definition}

\begin{remark}
\normalfont
Examples of compatibility with a grading of type ($A_{+}$) are provided by Theorem~\ref{struct}. 
\end{remark}

The following operation is the tool that we will need to piece together general nilpotent Lie algebras with 1-dimensional center 
from subalgebras of type~($A$) and Heisenberg algebras (Theorem~\ref{struct}). 

\begin{definition}\label{redsemi}
\normalfont
Let $\gg$ be a Lie algebra with $\dim\Zc(\gg)=1$ and with two subalgebras $\gg_1$ and $\gg_2$. 
We say that $\gg$ is the \emph{reduced semidirect product} of $\gg_1$ and $\gg_2$, 
and we write $\gg=\gg_1\redrtimes\gg_2$, if the following conditions are satisfied: 
\begin{itemize}
\item $[\gg_2,\gg_1]\subseteq\gg_1$; 
\item $\gg=\gg_1+\gg_2$ and $\Zc(\gg)=\gg_1\cap\gg_2$.
\end{itemize}
If also $[\gg_1,\gg_2]=0$, then we denote $\gg=\gg_1\redtimes\gg_2$ and we say that 
$\gg$ is the \emph{reduced direct product} of $\gg_1$ and $\gg_2$. 
\end{definition}

\begin{remark}
\normalfont
If $\gg=\gg_1\redrtimes\gg_2$ then $\gg_1\unlhd\gg$ and moreover we have a genuine semidirect product decomposition 
$\gg/\Zc(\gg)\simeq(\gg_1/\Zc(\gg))\rtimes(\gg_2/\Zc(\gg))$. 
\end{remark}

\begin{example}
\normalfont
For an arbitrary vector space $\Vc$ consider the associated Heisenberg algebra $\hg_{\Vc}=\KK\times\Vc^*\times\Vc$. 
Then for any vector subspaces $\Vc_1$ and $\Vc_2$ with $\Vc=\Vc_1\dotplus\Vc_2$ 
we have 
$\hg_{\Vc}=\hg_{\Vc_1}\redtimes\hg_{\Vc_2}$. 
\end{example}

\section{The generalization of Kirillov's lemma}\label{Sect3}

The main result of the present section is the following one, 
which will play a key role throughout the present paper in order to obtain 
our main results in Sections \ref{Sect4}--\ref{Sect5}, and 
is a sharpened version of the famous result known as Kirillov's lemma 
on the structure of nilpotent Lie algebras with 1-dimensional center 
(see \cite[Lemma 4.1]{Ki62} and Corollary~\ref{cor3.2} below).  

\begin{theorem}\label{K0}
Let $\gg$ be a Lie algebra with $\dim\Zc(\gg)=1$.   
Pick any nonzero $Z\in\Zc(\gg)$ and the linear subspaces $\Xg$ and $\Yg$ such that 
$$\Cc_2\gg=\Zc(\gg)\dotplus\Yg\text{ and }\gg=\Xg\dotplus\Cc(\Cc_2\gg:\gg)=\Xg\dotplus\Cc(\Yg:\gg).$$ 
Then the following assertions hold: 
\begin{enumerate}
\item\label{K0_item1} 
There exists a duality pairing $[\cdot,\cdot]\colon \Xg\times\Yg\to\Zc(\gg)\simeq\KK$ defined 
by the Lie bracket. 
Moreover $\dim(\gg/\Cc(\Cc_2\gg:\gg))=\dim(\Cc_2\gg/\Zc(\gg))$ 
and 
$$\dim\Cc_2\gg+\dim\Cc(\Cc_2\gg:\gg)=\dim\gg+1.$$ 
\item\label{K0_item2} 
If $Y_1,\dots,Y_m$ is a basis in $\Yg$, then there exist a unique basis $X_1,\dots,X_m$ in $\Xg$ 
and a unique basis $\gamma_1,\dots,\gamma_m$ in $\Cc(\Cc_2\gg:\gg)^\perp$ ($\subseteq\gg^*$) 
such that $[X_j,Y_k]=\delta_{jk}Z$ and $[X,Y_k]=\gamma_k(X)Z$ for all $j,k\in\{1,\dots,m\}$ and $X\in\gg$. 
\item\label{K0_item3} 
We have $[\gg,\gg]\subseteq\Cc(\Cc_2\gg:\gg)$ and $\Cc(\Cc_2\gg:\gg)$ is a characteristic ideal of $\gg$. 
\end{enumerate}
\end{theorem}

\begin{proof} 
First note that $\Xg=0$ if and only if $\Yg=0$. 
Indeed, one has 
$$\Xg=0\iff\gg=\Cc(\Cc_2\gg:\gg)\iff \Cc_2\gg\subseteq\Zc(\gg)\iff\Yg=0.$$ 
In this case the assertions in the statement hold true trivially. 
So let us assume from now on that $\Xg\ne 0$ and $\Yg\ne 0$.

For Assertion~\eqref{K0_item1} note that, 
since $\Cc_2\gg=\Yg\dotplus\Zc(\gg)$, we have  
$$[\Yg,\Xg]=[\Cc_2\gg,\Xg]\subseteq[\Cc_2\gg,\gg]\subseteq\Zc(\gg), $$
where the later inclusion holds true since 
$[\gg,[\gg,\Cc_2\gg]]=0$. 
Hence we have the well-defined mapping 
$[\cdot,\cdot]\colon\Xg\times\Yg\to\Zc(\gg)\simeq\KK $. 
This bilinear mapping is non-degenerate since 
$$\Cc(\Yg:\gg)\cap\Xg=\Cc(\Yg+\Zc(\gg):\gg)\cap\Xg=\Cc(\Cc_2\gg:\gg)\cap\Xg=0$$
and on the other hand 
$$\Cc(\Xg:\gg)\cap\Yg=\Cc(\Xg+\Cc(\Cc_2\gg:\gg):\gg)\cap\Yg=\Zc(\gg)\cap\Yg=0$$
where the first equality follows since $[\Cc(\Cc_2\gg:\gg),\Yg]\subseteq[\Cc(\Cc_2\gg:\gg),\Cc_2\gg]=0$. 
The equations involving dimensions are straightforward consequences of the assertion on the duality pairing. 

Assertion~\eqref{K0_item2} follows at once by using the duality pairing provided by Assertion~\eqref{K0_item1} 
and the fact that $[\gg,\Yg]=[\gg,\Yg+\Zc(\gg)]=[\gg,\Cc_2\gg]\subseteq\Zc(\gg)$ as above. 

For Assertion~\eqref{K0_item3} we use the Jacobi identity to obtain 
$$[[\gg,\gg],\Cc_2\gg]=[[\gg,\Cc_2\gg],\gg]+[[\Cc_2\gg,\gg],\gg]=0,$$ 
hence $[\gg,\gg]\subseteq\Cc(\Cc_2\gg:\gg)$. 
By the results of \cite[\S 1, no. 6]{Bo71} we also obtain that $\Cc(\Cc_2\gg:\gg)$ is a characteristic ideal of~$\gg$, 
and this completes the proof. 
\end{proof}

\begin{corollary}[{Kirillov's lemma \cite[Lemma 4.1]{Ki62}}]\label{cor3.2}
Let $\gg$ be a nilpotent Lie algebra with $\dim\Zc(\gg)=1$.   
Then for any nonzero $Z\in\Zc(\gg)$ 
and $Y\in \Cc_2\gg\setminus \Zc(\gg)$ 
there exists $X\in \gg$   
satisfying the the Heisenberg commutation relation $[X,Y]=Z$ and in addition $\gg=\gg_0\dotplus\RR X$ 
for the ideal $\gg_0:=\{V\in\gg\mid[V,Y]=0\}$. 
\end{corollary}

\begin{proof} 
Since $\gg$ is a nilpotent Lie algebra, one has $\Zc(\gg)\subsetneqq\Cc_2\gg$, 
hence we may use Theorem~\ref{K0} with $m\ge 1$, and we may choose $Y_1:=Y$, and $X:=X_1$. 
The fact that $\gg_0$ is an ideal of $\gg$ is a direct consequence of the hypothesis $Y\in \Cc_2\gg$. 
\end{proof}

\begin{remark}
\normalfont
Theorem~\ref{K0} can be regarded as a sharpened version of Kirillov's lemma for the following reason:  
The conclusion of Theorem~\ref{K0} provides not only a 3-dimensional Heisenberg subalgebra of $\gg$, 
but a $(2m+1)$-dimensional Heisenberg subalgebra $\spa\{X_1,\dots,X_m,Y_1,\dots,Y_m,Z\}$, 
where the integer $m\ge 1$ is the greatest integer satisfying 
$$[\gg,\spa\{Y_1,\dots,Y_m\}]\subseteq\Zc(\gg) 
\text{ and }
\spa\{Y_1,\dots,Y_m\}\cap \Zc(\gg)=0.$$
In fact, for ant $Y\not\in\spa\{Y_1,\dots,Y_m\}+\Zc(\gg)=\Cc_2\gg$ 
we have $[\gg,[\gg,Y]]\ne0$, that is, $[\gg,Y]\not\subset\Zc(\gg)$. 
\end{remark}

\begin{remark}
\normalfont
Assertion~\eqref{K0_item1} in the above Theorem~\ref{K0} can also be regarded 
as a special case of the following general phenomenon. 
Let $\Ac$ and $\Bc$ be finite-dimensional vector spaces and $\Psi\colon\Ac\times\Bc\to\KK$ 
be a bilinear form. 
If we denote 
$$\begin{aligned}
\Ac_0:=&\Ac^{\perp_\Psi}=\{a\in\Ac\mid(\forall b\in\Bc)\ \Psi(a,b)=0\}, \\
\Bc_0:=&\Bc^{\perp_\Psi}=\{b\in\Bc\mid(\forall a\in\Ac)\ \Psi(a,b)=0\}
\end{aligned}$$
then $\Psi$ gives rise to a duality pairing 
$$(\Ac/\Ac_0)\times(\Bc/\Bc_0)\to\KK,\quad (a+\Ac_0,b+\Bc_0)\mapsto\Psi(a,b).$$
In particular, in the setting of Theorem~\ref{K0}, we obtain an isomorphism of vector spaces  
$\gg/\Cc(\Cc_2\gg:\gg)\simeq(\Cc_2\gg/\Zc(\gg))^*$ provided that $\dim\Zc(\gg)=1$. 
\end{remark}

\subsection*{Some further properties of $\Cc_2\gg$}

\begin{proposition}\label{L0.4}
Let $\gg$ be a Lie algebra with   
a subalgebra~$\dg$ satisfying the conditions $[\dg,\dg]\subseteq\Zc(\gg)\subseteq\dg\subseteq\gg$.  
Then either $[\dg,\dg]=0$ or there exists a 2-step nilpotent subalgebra $\hg\subseteq\dg$ such that 
$\Zc(\hg)=\Zc(\gg)$, 
$\hg+\Zc(\dg)=\dg$, and $\hg\cap\Zc(\dg)=\Zc(\gg)$. 
\end{proposition}

\begin{proof}
Let us assume that $[\dg,\dg]\ne0$  
and let $\hg_0$ be any linear subspace of $\dg$ such that 
we have the direct sum decomposition 
$\hg_0\dotplus\Zc(\dg)= \dg$. 
Then $\hg:=\hg_0+\Zc(\gg)$ is an algebra that satisfies the conditions. 
In order to see this, note that $\Zc(\gg)\subseteq\Zc( \dg)$, hence $\hg_0\cap\Zc(\gg)=0$, 
and then we have the direct sum decomposition $\hg=\hg_0\dotplus\Zc(\gg)$. 
Since $\hg_0\dotplus\Zc( \dg)=\dg$, it then easily follows that 
$\hg+\Zc(\dg)=\dg$  and $\hg\cap\Zc(\dg)=\Zc(\gg)$. 

Moreover, 
$\hg$ is closed under the Lie bracket since 
$[\hg,\hg]\subseteq[\dg,\dg]=\Zc(\gg)\subseteq\hg$. 
Also $\hg$ is 2-step nilpotent, 
and we have $\Zc(\gg)\subseteq\Zc(\hg)\subseteq\hg=\hg_0\dotplus\Zc(\gg)$, 
hence $\Zc(\hg)=(\Zc(\hg)\cap \hg_0)\dotplus\Zc(\gg)$. 
On the other hand, since $\hg_0+\Zc(\dg)=\dg$ 
and $[\dg, \Zc(\dg)]=0$, it follows that $[\Zc(\hg)\cap \hg_0,\dg]=0$, 
that is, $\Zc(\hg)\cap \hg_0\subseteq\Zc(\dg)$. 
Since $\hg_0\cap\Zc(\dg)=0$, it then follows that $\Zc(\hg)\cap \hg_0=0$, 
hence $\Zc(\hg)=0\dotplus\Zc(\gg)=\Zc(\gg)$. 
This completes the proof. 
\end{proof}

\begin{corollary}\label{L0.5}
Let $\gg$ be a Lie algebra with $\dim\Zc(\gg)=1$. 
If $\dg$ is a subalgebra of $\gg$ such that $[\dg,\dg]\subseteq\Zc(\gg)$, 
then either $[\dg,\dg]=0$ or there exists a subalgebra $\hg\subseteq\dg$ such that $\hg$ is a Heisenberg algebra, 
$\hg+\Zc(\dg)=\dg$, and $\hg\cap\Zc(\dg)=\Zc(\gg)$. 
\end{corollary}

\begin{proof}
Since $[\dg,\dg]\subseteq\Zc(\gg)$ and $\dim\Zc(\gg)=1$, it follows that $\Zc(\gg)=[\dg,\dg]\subseteq\dg$. 
Hence we can use Proposition~\ref{L0.4} to obtain the subalgebra~$\hg$, 
and Remark~\ref{L0} ensures that $\hg$ is a Heisenberg algebra. 
\end{proof}

\begin{proposition}\label{L0.9}
If $\gg$ is a Lie algebra then the following assertions hold: 
\begin{enumerate}
\item
The ideal $\Cc_2\gg$ of $\gg$ is a 2-step nilpotent Lie algebra. 
\item 
Either $[\Cc_2\gg,\Cc_2\gg]=0$ or there exists a subalgebra $\hg$ such that 
$\Zc(\hg)=\Zc(\gg)$, 
$\hg+\Zc(\Cc_2\gg)=\Cc_2\gg$,  and $\hg\cap\Zc(\Cc_2\gg)=\Zc(\gg)$.
\end{enumerate}
\end{proposition}

\begin{proof}
The first part is well known and follows by  
$[\Cc_2\gg,[\Cc_2\gg,\Cc_2\gg]]\subseteq[\gg,[\gg,\Cc_2\gg]]=0$. 

For the second assertion we note that 
$[\Cc_2\gg,\Cc_2\gg]\subseteq[\gg,\Cc_2\gg]\subseteq\Zc(\gg)$, 
where the later inclusion follows since  $[\gg,[\gg,\Cc_2\gg]]=0$. 
We then see that Proposition~\ref{L0.4} can be applied with $\dg=\Cc_2\gg$, 
and the conclusion follows. 
\end{proof}

\begin{corollary}\label{L1}
If $\dim\Zc(\gg)=1$, then either $[\Cc_2\gg,\Cc_2\gg]=0$ or there exists a Heisenberg algebra $\hg$ such that 
$\hg+\Zc(\Cc_2\gg)=\Cc_2\gg$  and $\hg\cap\Zc(\Cc_2\gg)=\Zc(\gg)$. 
\end{corollary}

\begin{proof}
Use Proposition~\ref{L0.9} along with Remark~\ref{L0}. 
\end{proof}

For later use, we now provide a generalization of \cite[proof of Th. 3.1, Step~2]{Ra85}.

\begin{proposition}\label{L2}
Let $\gg$ be a Lie algebra with a subalgebra $\hg$ with   
$\Zc(\gg)\subseteq\hg\subseteq\Cc_2\gg$.  
Then the following assertions hold:  
\begin{enumerate}
\item\label{L2_item1} 
We have $[\gg,\hg]\subseteq\Zc(\gg)$ and $\hg\unlhd\gg$.    
\item\label{L2_item4}  
We have $\Zc(\Cc_2\gg)\subseteq\Cc_2(\Cc(\hg:\gg))$. 
If we moreover assume $\hg+\Zc(\Cc_2\gg)=\Cc_2\gg$, 
then we actually have $\Zc(\Cc_2\gg)\subseteq\Cc_2(\Cc(\hg:\gg))\subseteq\Cc(\Cc_2\gg:\gg)$.
\item\label{L2_item5} If $\hg+\Zc(\Cc_2\gg)=\Cc_2\gg$ and $\hg+\Cc(\hg:\gg)=\gg$, 
then $\Zc(\Cc_2\gg)=\Cc_2(\Cc(\hg:\gg))$. 
\item\label{L2_item2}  
If $\dim\Zc(\gg)=1$ and $\hg$ is a Heisenberg algebra, then $\hg+\Cc(\hg:\gg)=\gg$  
and $\hg\cap\Cc(\hg:\gg)=\Zc(\gg)=\Zc(\Cc(\hg:\gg))$. 
\end{enumerate}
\end{proposition}

\begin{proof}
For Assertion~\eqref{L2_item1} 
note that $[\gg,[\gg,\hg]]\subseteq[\gg,[\gg,\Cc_2\gg]]=0$. 
Therefore $[\gg,\hg]\subseteq\Zc(\gg)\subseteq\hg$, hence $\hg\unlhd\gg$.

For the first part of Assertion~\eqref{L2_item4}  
note that 
$[\Zc(\Cc_2\gg),\hg]\subseteq[\Zc(\Cc_2\gg),\Cc_2\gg]=0$, 
hence $\Zc(\Cc_2\gg)\subseteq\Cc(\hg:\gg)$. 
Moreover, 
$[\Cc(\hg:\gg),[\Cc(\hg:\gg),\Zc(\Cc_2\gg)]]\subseteq[\gg,[\gg,\Cc_2\gg]]=0$, 
hence we have $\Zc(\Cc_2\gg)\subseteq\Cc_2(\Cc(\hg:\gg))$. 
For the second part of Assertion \eqref{L2_item4} 
let us assume that we have $\hg+\Zc(\Cc_2\gg)=\Cc_2\gg$. 
Then 
\begin{align}
[\Cc_2(\Cc(\hg:\gg)),\Cc_2\gg]
&=[\Cc_2(\Cc(\hg:\gg)),\hg+\Zc(\Cc_2\gg)] \nonumber\\
&\subseteq [\Cc_2(\Cc(\hg:\gg)),\hg] +[\Cc_2(\Cc(\hg:\gg)),\Zc(\Cc_2\gg)] \nonumber\\
&\subseteq [\Cc(\hg:\gg),\hg] +[\Cc_2\gg,\Zc(\Cc_2\gg)] \nonumber\\
&=0\nonumber
\end{align}
that is, $\Cc_2(\Cc(\hg:\gg))\subseteq\Cc(\Cc_2\gg:\gg)$. 

For Assertion~\eqref{L2_item5} it suffices to prove the inclusion $\supseteq$ 
since the converse inclusion follows by Assertion \eqref{L2_item4}. 
Let $X\in\Cc_2(\Cc(\hg:\gg))$ arbitrary. 
Again by Assertion \eqref{L2_item4} we have $X\in \Cc(\Cc_2\gg:\gg)$ hence, 
in order to prove that $X\in\Zc(\Cc_2\gg)$, it suffices to check that $X\in\Cc_2\gg$. 
And this is indeed the case since by using the hypothesis $\hg+\Cc(\hg:\gg)=\gg$ we obtain 
$$\begin{aligned}{}
[\gg,[\gg,X]]
&=[\hg+\Cc(\hg:\gg),[\hg+\Cc(\hg:\gg),X]] \\
&=[\hg+\Cc(\hg:\gg),[\hg,X]]
+[\hg,[\Cc(\hg:\gg),X]]
+[\Cc(\hg:\gg),[\Cc(\hg:\gg),X]]
\end{aligned}$$
and here the first term vanishes since $[\gg,X]\subseteq\Zc(\gg)$ by Assertion~\eqref{L2_item1}, 
the second term vanishes since $X\in\Cc(\hg:\gg)$ hence $[\Cc(\hg:\gg),X]\subseteq\Cc(\hg:\gg)$, 
and the third vanishes since actually $X\in\Cc_2(\Cc(\hg:\gg))$. 

For the first equality in Assertion \eqref{L2_item2} it suffices to prove the inclusion $\supseteq$. 
Hence for arbitrary $X\in\gg$ we have to prove that $X\in \hg+\Cc(\hg:\gg)$. 
To this end let us pick $Z_0\in\Zc(\gg)$ with $Z_0\ne0$, so that $\Zc(\gg)=\KK Z_0$. 
We have seen above that $[\gg,\hg]\subseteq\Zc(\gg)$, hence $[X,\hg]\subseteq\KK Z_0$, 
and then there exists $\lambda\in\hg^*$ such that for every $Y\in\hg$ we have $[X,Y]=\lambda(Y)Z_0$. 
On the other hand, 
since $\lambda\in\hg^*$ and $\hg$ is a finite-dimensional Heisenberg algebra with the center equal to $\KK Z_0$, 
it is easily checked that there exists $V\in\hg$ such that for every $Y\in\hg$ we have $[V,Y]=\lambda(Y)Z_0$. 
Therefore $[X-V,\hg]=0$, and then $X\in V+\Cc(\hg:\gg)\subseteq\hg+\Cc(\hg:\gg)$. 

Furthermore, 
it is obvious that $\hg\cap\Cc(\hg:\gg)=\Zc(\hg)$, 
and we have $\Zc(\hg)=\Zc(\gg)$ directly by the hypothesis. 
For the last equality, 
note that $[\Zc(\Cc(\hg:\gg)),\hg]\subseteq[\Cc(\hg:\gg),\hg]=0$ 
and $[\Zc(\Cc(\hg:\gg)),\Cc(\hg:\gg)]=0$. 
Then by Assertion \eqref{L2_item2} we get $[\Zc(\Cc(\hg:\gg)),\gg]=0$, that is, 
$\Zc(\Cc(\hg:\gg))\subseteq\Zc(\gg)$. 
The converse to this inclusion is obvious, hence we obtain 
$\Zc(\Cc(\hg:\gg))=\Zc(\gg)$, 
and this completes the proof. 
\end{proof}

\section{Structure theory for algebras of type ($A_{+}$)}\label{Sect4}

Here is the main result of this section, 
which establishes the precise relationship between the notions introduced in Definition~\ref{def_A}. 

\begin{theorem}\label{new1}
If $\kg$ is a 3-step nilpotent Lie algebra with $\dim\Zc(\kg)=1$, 
then the following assertions are equivalent: 
\begin{enumerate}
\item\label{new1_item1}
The Lie algebra $\kg$ is of type ($A_{+}$). 
\item\label{new1_item2}
The Lie algebra $\kg$ is of type ($A$) and admits a grading of type ($A_{+}$).
\end{enumerate}
\end{theorem}

We postpone the proof of this theorem for the moment, 
in order to discuss some of its implications. 
For the following corollary we recall that a Lie algebra is called \emph{characteristically nilpotent} 
if every derivation of that algebra is a nilpotent map. 
(See for instance \cite{GK00} and \cite{Bu06} for more details on this class of Lie algebras.)
Hence Engel's theorem implies that every finite-dimensional Lie algebra of this type 
is in particular a nilpotent Lie algebra. 

\begin{corollary}\label{new2}
Any 3-step nilpotent Lie algebra of type $(A_{+})$ 
has an invertible derivation, hence it 
cannot be characteristically nilpotent. 
\end{corollary}

\begin{proof}
Let $\kg$ be a Lie algebra of type $(A_{+})$. 
It follows by Theorem~\ref{new1} that there exists a grading  $\gamma=(\zg,\cg,\Vc)$ 
of type ($A_{+}$) of~$\kg$, 
hence we have $\kg=\zg\dotplus\cg\dotplus\Vc$ with $[\zg,\kg]=[\cg,\cg]=0$, $[\cg,\Vc]\subseteq\zg$ and $[\Vc,\Vc]\subseteq\cg$. 
Then the construction of an invertible derivation of $\kg$ is well known. 
Namely, define the linear map $D\colon\kg\to\kg$ with the properties $\Ker (D-3\id)=\zg$, $\cg=\Ker(D-2\id)$, 
and $\Vc=\Ker(D-\id)$, then it is easily checked that for all $X_1,X_2\in\kg$ 
we have $D[X_1,X_2]=[DX_1,X_2]+[X_1,DX_2]$. 
Hence $\kg$ has an invertible derivation, 
and then it cannot be characteristically nilpotent.  
\end{proof}

\begin{remark}
\normalfont 
In connection with the above corollary, 
it would be interesting to know whether there exists any 3-step nilpotent Lie algebra with 1-dimensional center 
which is characteristically nilpotent.  
For instance, the center of the 3-step nilpotent, characteristically nilpotent, 8-dimensional Lie algebra 
constructed in \cite{DL57} is 2-dimensional. 
\end{remark}

\begin{corollary}\label{new4}
For any integer $n\ge 1$ let $A_{+}^n$ denote the set of all skew-symmetric bilinear maps 
$\psi\colon\KK^n\times\KK^n\to\KK^n$ satisfying the condition 
$$(\forall x,y,z\in\KK^n)\quad \langle\psi(x,y),z\rangle+\langle\psi(y,z),x\rangle+\langle\psi(z,x),y\rangle=0 $$
where $\langle\cdot,\cdot\rangle\colon \KK^n\times\KK^n\to\KK$ is the canonical $\KK$-bilinear scalar product. 
Consider the natural action of the group $\GL(n,\KK)$ on $A_{+}^n$ given by 
$$(g.\psi)(x,y)=(g^\top)^{-1}(\psi(g^{-1}x,g^{-1}y))$$ 
for all $x,y\in\KK^n$, $\psi\in A_{+}^n$ and $g\in\GL(n,\KK)$. 
Then there exists a one-to-one correspondence between 
the isomorphism classes of 3-step nilpotent Lie algebras of type $(A_{+})$ 
and the $\GL(n,\KK)$-orbits in $A_{+}^n$. 
\end{corollary}

\begin{proof}
If $\kg$ is a $(2n+1)$-dimensional 3-step nilpotent Lie algebra of type $(A_{+})$, 
then Theorem~\ref{new1} provides a grading  $\gamma=(\zg,\cg,\Vc)$ 
of type ($A_{+}$), with $\dim\Vc=\dim\cg=n$, and the Lie bracket of $\kg$ 
is uniquely determined by its restriction $\psi:=[\cdot,\cdot]\colon\Vc\times\Vc\to\cg$. 
By choosing a basis in $\Vc$, and the dual basis in $\cg$ via the duality pairing defined by the Lie bracket 
$\cg\times\Vc\to\zg$, we thus obtain $\psi\in A_{+}^n$. 

Conversely, every $\psi\in A_{+}^n$ turns $\KK^n\times\KK^n\times\KK$ into a 3-step nilpotent Lie algebra of type $(A_{+})$ 
with a grading $\gamma=(\KK,\KK^n,\KK^n)$ of type ($A_{+}$), 
and the action of the group $\GL(n,\KK)$ on $A_{+}^n$ accounts for the isomorphisms of Lie algebras of type $(A_{+})$  
obtained in this way. 
\end{proof}

\begin{corollary}\label{new5}
Every  3-step nilpotent Lie algebra of type $(A_{+})$ is a nontrivial degeneration of another Lie algebra.  
\end{corollary}

\begin{proof}
We have seen in the proof of Corollary~\ref{new2} that every 3-step nilpotent Lie algebra of type $(A_{+})$ 
has an invertible diagonalizable derivation, and then one can use the method of proof of \cite[Prop. 2.1]{HT13}.  
\end{proof}

\begin{remark}
\normalfont 
The above corollary shows that the Grunewald-O'Halloran conjecture raised in \cite{GH93} 
holds true for the 3-step nilpotent Lie algebras of type $(A_{+})$; 
see also \cite{HT13}. 
\end{remark}

We now turn to the proof of Theorem~\ref{new1}, 
which will require two lemmas. 
The main assertions of the following lemma can be found at least implicitly in \cite[proof of Th. 3.1, Step~3]{Ra85}.

\begin{lemma}\label{L3}
Let $\kg$ be a 3-step nilpotent Lie algebra of type~($A$).
If we define $\bg:=\Cc(\Cc_2\kg:\kg)$, then the following assertions hold:  
\begin{enumerate}
\item\label{L3_item1}
We have $[\bg,\bg]\subseteq\Zc(\kg)$ and $[\bg,[\bg,\bg]]=0$. 
\item\label{L3_item2} 
Either $[\bg,\bg]=0$ or there exists a Heisenberg algebra $\hg'$ with $\hg'+\Zc(\bg)=\bg$ and $\hg'\cap\Zc(\bg)=\Zc(\kg)$. 
\item\label{L3_item3b} 
For any direct sum decomposition $\Cc_2\kg=\cg\dotplus\Zc(\kg)$
there exists a direct sum decomposition $\Vc\dotplus\bg=\kg$ for which  
$[\Vc,\Vc]\subseteq\cg$. 
\end{enumerate}
\end{lemma}

\begin{proof} 
For Assertion~\eqref{L3_item1} note that 
$[[\bg,\bg],\kg]=[\bg,[\bg,\kg]]\subseteq[\bg,\Cc_2\kg]=0$
where the inclusion follows since $\bg\subseteq\kg$ and $[\kg,[\kg,[\kg,\kg]]]=0$, 
and the last equality follows by the definition of $\bg$. 
It then follows that $[\bg,\bg]\subseteq\Zc(\kg)$ and $[[\bg,\bg],\bg]=0$. 

For Assertion~\eqref{L3_item2} just note that Corollary~\ref{L0.5} can be applied with $\gg=\kg$ and $\dg=\bg$, 
since we have just seen that $[\bg,\bg]\subseteq\Zc(\kg)$. 

It remains to prove Assertion~\eqref{L3_item3b}. 
Let the decomposition $\Cc_2\kg=\cg\dotplus\Zc(\kg)$ be fixed. 
By starting from an arbitrary decomposition $\Vc\dotplus\bg=\kg$, 
we will show how one can construct a new decomposition $\widetilde{\Vc}\dotplus\bg=\kg$ 
with the additional property $[\widetilde{\Vc},\widetilde{\Vc}]\subseteq\cg$. 
Since $[\kg,[\kg,[\kg,\kg]]]=0$, it follows that $[\kg,\kg]\subseteq\Cc_2\kg$, 
hence in particular $[\Vc,\Vc]\subseteq\Cc_2\kg=\cg\dotplus\Zc(\kg)$. 
It follows that there exist uniquely determined bilinear mappings 
$Y\colon\Vc\times\Vc\to\cg$ and $Z\colon\Vc\times\Vc\to\Zc(\kg)$ such that 
$$(\forall V,V'\in\Vc)\quad [V,V']=Y(V,V')+Z(V,V').$$
By using Theorem~\ref{K0}\eqref{K0_item1}, we see that 
there exists a uniquely determined linear mapping $\Vc\to\cg$, $V\mapsto C_V$ such that 
for every $V\in\Vc$ we have $Z(V,\cdot)=[C_V,\cdot]$ on~$\Vc$. 
Then it is easily checked that 
$$(\forall V,V'\in\Vc)\quad \Bigl[V-\frac{1}{2}C_V,V'-\frac{1}{2}C_{V'}\Bigr]=Y(V,V')$$
(where we may use $\frac{1}{2}\in\KK$ since the characteristic of $\KK$ is different from~2)
and it follows that 
if we define 
$$\widetilde{\Vc}:=\Bigl\{V-\frac{1}{2}C_V\mid V\in\Vc\Bigr\}$$
then this is a linear subspace of $\kg$ with the properties
$\widetilde{\Vc}\dotplus\bg=\kg$ and $[\widetilde{\Vc},\widetilde{\Vc}]\subseteq\cg$. 
This completes the proof. 
\end{proof}

The following lemma was suggested by \cite[proof of Th. 3.1, Step~4]{Ra85}.

\begin{lemma}\label{L4}
Let $\kg$ be a 3-step nilpotent Lie algebra of type~($A$). 
Set $\bg:=\Cc(\Cc_2\kg:\kg)$ and assume that 
for some linear subspaces $\cg$ and $\Vc$ of $\kg$ we have 
$\Cc_2\kg=\cg\dotplus\Zc(\kg)$ and $\Vc\dotplus\bg=\kg$, 
with $[\Vc,\Vc]\subseteq\cg$. 
If we also define 
$\bg_1:=\{X\in\bg\mid[X,\Vc]\subseteq\cg\}$ and 
$\gg_1:=\Vc+\Cc_2\kg$, 
then the following assertions hold: 
\begin{enumerate}
\item\label{L4_item1}
The linear subspace $\gg_1$ is a subalgebra of $\kg$, 
we have $\gg_1=\Vc\dotplus\Cc_2\kg$, 
and moreover $\Cc_2\kg\subseteq\bg$. 
\item\label{L4_item2} 
We have $\bg=\bg_1\dotplus\cg$ and $\cg=\{X\in\bg\mid[X,\Vc]\subseteq\Zc(\kg)\}$. 
\item\label{L4_item3} 
We have $\gg_1+\bg_1=\kg$. 
\item\label{L4_item4} 
We have $\gg_1\cap\bg_1=\Zc(\gg_1)=\Zc(\gg)$.  
\end{enumerate}
\end{lemma}

\begin{proof}
For Assertion~\eqref{L4_item1} note that 
$[\Vc,\Vc]\subseteq\cg\subseteq\Cc_2\kg$. 
Since $\Cc_2\kg$ is an ideal of $\kg$, 
it then follows that for $\gg_1=\Vc+\Cc_2\kg$ we have $[\gg_1,\gg_1]\subseteq\Cc_2\kg\subseteq\gg_1$, 
hence $\gg_1$ is indeed a subalgebra of $\kg$. 
Moreover, since $[\Cc_2\kg,\Cc_2\kg]=0$, it follows that $\Cc_2\kg\subseteq\bg$, 
and in particular $\Cc_2\kg\cap\Vc=0$. 
Therefore $\gg_1=\Vc\dotplus\Cc_2\kg$. 

For Assertion~\eqref{L4_item2} let us first see that 
$\bg_1\cap\cg=0$. 
In fact, if $X\in\bg_1\cap\cg$, 
then  by using Theorem~\ref{K0}\eqref{K0_item1} we get $[X,\Vc]\subseteq\cg\cap\Zc(\kg)=0$, 
hence by Theorem~\ref{K0}\eqref{K0_item1} again we get $X=0$. 
In order to see that $\bg_1+\cg=\bg$, let $X\in\bg$ arbitrary. 
Then 
$$[X,\Vc]\subseteq[\bg,\Vc]\subseteq[\kg,\kg]\subseteq\Cc_2\kg=\cg\dotplus\Zc(\kg)$$
where the second inclusion follows since $[\kg,[\kg,[\kg,\kg]]]=0$. 
By using the duality pairing referred to in Theorem~\ref{K0}\eqref{K0_item1}, 
it then easily follows that there exists $Y\in\cg$ such that $[X-Y,\Vc]\subseteq\cg$. 
Then $X-Y\in\bg_1$, hence $X\in Y+\bg_1\subseteq\cg+\bg_1$. 
Thus $\bg=\cg+\bg_1$, and this is actually a direct sum decomposition, 
and this directly implies $\cg=\{X\in\bg\mid[X,\Vc]\subseteq\Zc(\kg)\}$. 

Assertion~\eqref{L4_item3} follows since 
$$\kg\supseteq\bg_1+\gg_1=\bg_1+\Vc+\Cc_2\kg\supseteq\bg_1+\Vc+\cg=\bg+\Vc=\kg$$
where the next-to-last equality relies on Assertion~\eqref{L4_item2}. 

It remains to prove Assertion~\eqref{L4_item4}. 
It is clear that $\gg_1\cap\bg_1\supseteq\Zc(\gg_1)\supseteq\Zc(\kg)$. 
Also, by the second part of Assertion~\eqref{L4_item1} 
we get $\gg_1\cap\bg_1=\{X\in\Cc_2\kg\mid[X,\Vc]\subseteq\cg\}$. 
On the other hand, since $[\kg,[\kg,\Cc_2\kg]]=0$, 
we have $[\Cc_2\kg,\kg]\subseteq\Zc(\kg)$, hence 
$$\gg_1\cap\bg_1=\{X\in\Cc_2\kg\mid[X,\Vc]=0\}=\Zc(\gg_1)$$
where the later equality follows since $\gg_1=\Vc+\Cc_2\kg$ and $[\Cc_2\kg,\Cc_2\kg]=0$ by hypothesis. 
Moreover, if $X\in\Cc_2\kg$ and $[X,\Vc]=0$, then by using again the equality 
$[\Cc_2\kg,\Cc_2\kg]=0$ along with Assertion~\eqref{L4_item2} we get $[X,\bg]=[X,\Vc]=0$, 
hence $[X,\kg]=[X,\bg+\Vc]=0$, and then $X\in\Zc(\kg)$. 
Consequently $\gg_1\cap\bg_1=\Zc(\gg_1)\subseteq\Zc(\kg)$, 
and this concludes the proof. 
\end{proof}

\begin{proof}[Proof of Theorem~\ref{new1}]
``$\eqref{new1_item1}\Rightarrow\eqref{new1_item2}$'' 
If the Lie algebra $\kg$ is of type ($A_{+}$), then $\Cc_2\kg$ is a maximal abelian subalgebra, 
hence $\Cc_2\kg=\Cc(\Cc_2\kg:\kg)$. 
Then in Lemma~\ref{L3} we have $\bg=\Cc_2\kg=\cg+\Zc(\kg)$, and the conclusion follows directly. 

``$\eqref{new1_item1}\Rightarrow\eqref{new1_item2}$'' 
Let $\gamma=(\zg,\cg,\Vc)$ be a grading of type ($A_{+}$) of~$\kg$. 
It is clear that $\zg+\cg\subseteq\Cc_2\kg$, hence by using the direct sum decomposition 
$\kg=(\zg\dotplus\cg)\dotplus\Vc$ 
we get 
$\Cc_2\kg=(\zg\dotplus\cg)\dotplus(\Cc_2\kg\cap\Vc)$. 
Now, since we have the duality pairing $[\cdot,\cdot]\colon\cg\times\Vc\to\zg$ and moreover $[\Cc_2\kg,\Cc_2\kg]=0$ by the hypothesis, 
we easily get $\Cc_2\kg\cap\Vc=0$, hence $\Cc_2\kg=\zg\dotplus\cg$. 

On the other hand, by using once again the decomposition $\kg=(\zg\dotplus\cg)\dotplus\Vc$ 
along with the duality pairing $[\cdot,\cdot]\colon\cg\times\Vc\to\zg$, 
it easily follows that $\zg+\cg$ is a maximal abelian subalgebra of $\kg$. 
We thus see that the Lie algebra $\kg$ is of type ($A_{+}$). 
\end{proof}

\section{Applications of reduced semidirect products}\label{Sect5}

In this section we use the notion of reduced semidirect product (Definition~\ref{redsemi}) 
and the results established so far in order to provide some insight on the importance 
of the nilpotent Lie algebras of type ($A$) and ($A_{+}$) for the whole class of nilpotent Lie algebras 
with 1-dimensional center. 
Note that the first assertion of the following theorem does not require the hypothesis of 3-step nilpotency. 

\begin{theorem}\label{struct}
\begin{enumerate}
\item\label{struct_item1} 
Every nilpotent Lie algebra with 1-dimensional center is the reduced direct product of a Heisenberg algebra and a Lie algebra of type~($A$). 
\item\label{struct_item2} 
Every 3-step nilpotent Lie algebra of type ($A$) is the reduced semidirect product of a Lie subalgebra of type ($A_{+}$) 
and a 2-step nilpotent Lie subalgebra whose derived algebra is contained in the center  
and which is compatible with a suitable grading of type ($A_{+}$) of the other subalgebra. 
\end{enumerate}
\end{theorem}

\begin{proof}
\eqref{struct_item1} 
Let $\gg$ be a nilpotent Lie algebra with 1-dimensional center. 
If $[\Cc_2\gg,\Cc_2\gg]=0$, then $\gg$ is of type~($A$) and there is nothing to prove. 

If however $[\Cc_2\gg,\Cc_2\gg]\ne0$, then Corollary~\ref{L1}
provides a Heisenberg algebra $\hg$ such that $\hg+\Zc(\Cc_2\gg)=\Cc_2\gg$  and $\hg\cap\Zc(\Cc_2\gg)=\Zc(\gg)$. 
Then by Proposition~\ref{L2} we get $\hg+\Cc(\hg:\gg)=\gg$ and $\hg\cap\Cc(\hg:\gg)=\Zc(\gg)$. 
Since moreover $[\hg,\Cc(\hg:\gg)]=0$, it follows 
that $\gg$ is equal to the reduced direct product $\hg\redtimes\Cc(\hg:\gg)$. 
Moreover, it follows by Proposition~\ref{L2}(\eqref{L2_item4}--\eqref{L2_item2}) 
that $\Cc(\hg:\gg)$ is a Lie algebra of type~($A$). 

\eqref{struct_item2} 
Let $\kg$ be a Lie algebra of type ($A$). 
If $\Cc_2\kg$ is a maximal abelian subalgebra of $\kg$, then $\kg$ is a Lie algebra of type ($A_{+}$) 
and there is nothing else to do. 

Now assume that $\Cc_2\kg$ fails to be a maximal abelian subalgebra of $\kg$, and  
denote $\zg:=\Zc(\kg)$ and $\bg:=\Cc(\Cc_2\kg:\kg)$, so that $\Cc_2\kg\varsubsetneqq\bg$. 
Also pick any linear subspace $\cg$ that gives rise to a direct sum decomposition $\Cc_2\kg=\cg\dotplus\zg$. 
Theorem~\ref{K0} and Lemma~\ref{L3} ensure that there exists a linear subspace $\Vc$ with $\kg=\Vc\dotplus\bg$, $[\Vc,\Vc]\subseteq\cg$, 
and such that the Lie bracket gives a duality pairing $[\cdot,\cdot]\colon\cg\times\Vc\to\zg$. 
Then define $\bg_1:=\{X\in\bg\mid[X,\Vc]\subseteq\cg\}$ and $\gg_1:=\zg+\cg+\Vc=\Cc_2\kg+\Vc$, 
so that $\gamma=(\zg,\cg,\Vc)$ is a grading of type~($A_{+}$) for the Lie algebra~$\gg_1$.
It also follows by Lemma~\ref{L4} that $\bg=\bg_1\dotplus\cg$, $\kg=\gg_1+\bg_1$ and $\gg_1\cap\bg_1=\zg$. 

We have $[[\bg_1,\bg_1],\Vc]\subseteq[[\bg_1,\Vc],\bg_1]\subseteq[\cg,\bg_1]\subseteq[\Cc_2\gg,\bg]=0$, 
hence $[\bg_1,\bg_1]\subseteq\bg_1$, and this shows that $\bg_1$ is a subalgebra of $\kg$. 
Moreover $[\kg,\gg_1]=[\gg_1+\bg_1,\gg_1]\subseteq\gg_1$, and thus $\gg_1$ is an ideal of $\kg$. 
Therefore 
$\kg$ is isomorphic to the reduced semidirect product $\gg_1\redrtimes\bg_1$. 
This homomorphism is compatible with the grading $\gamma$ since $[\bg_1,\Vc]\subseteq\cg$ and 
$[\bg_1,\zg+\cg]\subseteq[\bg,\Cc_2\kg]=0$. 
Moreover, $\bg_1$ is a 2-step nilpotent Lie algebra by Lemma~\ref{L3}\eqref{L3_item1} since 
$\bg_1\subseteq\bg$.  
This completes the proof. 
\end{proof}

The following theorem indicates precisely which ones of the reduced semidirect products that occur in Theorem~\ref{struct}\eqref{struct_item2} for $\KK=\RR$ give rise to Lie groups with flat generic coadjoint orbits. 
We refer to \cite{CG90} for the terminology used here, 
which is related to the method of coadjoint orbits in the representation theory. 

\begin{theorem}\label{flat}
Let $\KK=\RR$ and $\gg=\kg\redrtimes\ag$ be a reduced semidirect product with the 1-dimensional center~$\zg$, 
where the Lie algebra~$\kg$ admits a grading $\gamma=(\zg,\cg,\Vc)$ of type ($A_{+}$), 
and let $\ag$ be a nilpotent Lie algebra with $[\ag,\ag]\subseteq\zg$ 
and which is compatible with the grading~$\gamma$.  
Then the following assertions hold: 
\begin{enumerate}
\item\label{flat_item1} 
The Lie group associated with $\kg\redrtimes\ag$ 
has flat generic coadjoint orbits if and only if either $\ag=\zg$ or $\ag$ is a Heisenberg algebra. 
\item\label{flat_item2} 
If this is the case, then there is a maximal abelian ideal of $\kg\redrtimes\ag$ 
which is a polarization at every point of a generic coadjoint orbit.
\end{enumerate}
\end{theorem}

\begin{proof}
First recall that the hypothesis $\gg=\kg\redrtimes\ag$ implies in particular $\zg=\kg\cap\ag$. 
Pick any $0\ne Z_0\in\zg$ and let $\ag_0$ be any linear subspace of $\ag$ such that $\ag=\zg\dotplus\ag_0$. 
Then we have $\gg=\zg\dotplus\cg\dotplus\ag_0\dotplus\Vc$, 
so there exists a unique 
$\xi_0\in\gg^*$ with $\langle \xi_0,Z_0\rangle=1$ and $\Ker\xi_0=\cg\dotplus\ag_0\dotplus\Vc$.  

\eqref{flat_item1} 
Since $\xi_0\vert_{\zg}\not\equiv 0$, it follows by \cite[Prop. 4.5.7]{CG90} 
that the coadjoint orbit of $\xi_0$ has maximal dimension. 
We have $[\ag_0,\Vc]\subseteq\cg$, $[\ag_0,\cg+\zg]=0$ and $[\ag_0,\hg]=0$, 
hence $[\ag_0,\gg]\subseteq\cg\subseteq\Ker\xi_0$. 
This shows that $\ag_0\subseteq\gg_{\xi_0}$. 
Since $Z_0\not\in\ag_0$, it follows that $\zg\cap\ag_0=0$, and then the above inclusion relation shows that 
if the coadjoint orbit of $\xi_0$ is flat and of maximal dimension, then necessarily $\ag_0=0$. 
By using Corollary~\ref{L0.5} we then obtain that either $\ag=\zg$ or $\ag$ is a Heisenberg algebra. 

Conversely, assume that $\ag$ is a Heisenberg algebra. 
We will prove that $\gg_{\xi_0}=\zg$. 
To this end assume $\zg\subsetneqq\gg_{\xi_0}$ and pick any $X\in\gg_{\xi_0}\setminus\zg$. 
In particular $X\ne 0$ and there exist $Z\in\zg,C\in\cg$, $V\in\Vc$, and $H\in\ag\setminus(\RR^*Z_0)$  
such that $X=Z+C+V+H$. 

If $H\ne 0$ then $H\in\ag\setminus\zg$. 
Since $\ag$ is a Heisenberg algebra, there exists $H_1\in\ag$ with $0\neq[H,H_1]\in\zg$, 
hence $[H_1,X]=[H_1,V]+[H_1,H]$ with $[H_1,V]\in[\ag,\Vc]\subseteq\cg$ 
by the compatibility of $\ag$ with the grading~$\gamma$. 
Therefore $\langle\xi_0,[X,H_1]\rangle=\langle\xi_0,[H,H_1]\rangle\ne 0$, 
and in particular $\langle\xi_0,[X,\gg]\rangle\ne0$, which contradicts the choice $X\in\gg_{\xi_0}$. 
Consequently $H=0$. 

If $V\ne0$, then by Theorem~\ref{K0}\eqref{K0_item1} 
there exists $C_1\in\cg$ with $0\neq[C_1,V]\in\zg$. 
We now have $[C_1,X]=[C_1,V]$ hence $\langle\xi_0,[X,C_1]\rangle=\langle\xi_0,[V,C_1]\rangle\ne0$, 
and then we obtain again $\langle\xi_0,[X,\gg]\rangle\ne0$, 
which contradicts the choice $X\in\gg_{\xi_0}$. 
Therefore also $V=0$. 

If at least $C\ne0$, then again by Theorem~\ref{K0}\eqref{K0_item1} there exists an element $V_1\in\Vc$ with $[V_1,C]\ne0$, 
hence $\langle\xi_0,[X,V_1]\rangle=\langle\xi_0,[C,V_1]\rangle\ne0$. 
Therefore, just as above, $\langle\xi_0,[X,\gg]\rangle\ne0$, 
which once again contradicts the choice $X\in\gg_{\xi_0}$. 

Consequently we must have $H=V=C=0$, and then $X=Z\in\zg$, 
which shows that $\gg_{\xi_0}=\zg$.

\eqref{flat_item2} 
For the second assertion note that the abelian ideal $\cg+\zg$ together with any polarization of $\ag$ span 
an abelian ideal polarization at any functional $\xi_0$ chosen as above. 
Since $[\cdot,\cdot]\colon\cg\times\Vc\to\zg\simeq\RR$ 
is a nondegenerate bilinear map, it is easily checked that the above abelian ideal polarization 
is a maximal abelian ideal of~$\gg$.  
This completes the proof. 
\end{proof}

\begin{corollary}\label{new3}
Let $\KK=\RR$ and $\kg$ be a 3-step nilpotent Lie algebra with 1-dim\-ensional center. 
Then $\kg$ is of type ($A_+$) if and only if it has flat generic coadjoint orbits and 
$\Cc_2\kg$ is a polarization at every point of these orbits. 
\end{corollary}

\begin{proof}
Use Theorems \ref{struct} and \ref{flat}. 
\end{proof}

\begin{remark}
\normalfont
It follows by Theorems \ref{struct} and \ref{flat} that 
every 3-step nilpotent Lie group with 1-dimensional center and with flat generic coadjoint orbits 
is a special nilpotent Lie group in the sense of \cite{Co83}. 
We also mention \cite[Prop. 2.2]{EO03}, which proves by a case-by-case analysis that 
there exist commutative ideal polarizations as in Theorem~\ref{flat}\eqref{flat_item2} for every 
(complex) nilpotent Lie algebra of dimension $\le 7$ with flat generic coadjoint orbits. 

In this context we recall from \cite[Cor. 3.2]{MR91} that every ideal polarization is necessarily abelian. 
In fact, as noted in \cite[Lemma 2.1(1.)]{BC13} even the isotropic ideals of any symplectic Lie algebra are abelian. 
\end{remark}

\section{Examples}\label{Sect6}

In Examples \ref{N4N1}--\ref{N6N6} below 
we assume $\KK=\RR$ and discuss all the examples of 3-step nilpotent Lie algebras of dimension $\le 6$, 
having 1-dimensional centers; 
they correspond to the items N4N1, N5N3, N5N6, N6N1, N6N4, N6N5, and N6N6 from \cite{Pe88}. 
(See also \cite{CGS12} for the classification of the 6-dimensional nilpotent Lie algebras over an arbitrary field.)
It is remarkable that all of the possible situations from the above Theorem~\ref{struct} 
can be illustrated by these low dimensional examples, as we will point out below. 

\begin{example}[Type~($A$)]\label{N4N1} 
\normalfont
Let $\gg=\spa\{X_1,X_2,X_3,X_4\}$ be the 4-dimensional filiform Lie algebra with 
$$[X_4,X_3]=X_2,\ [X_4,X_2]=X_1. $$
Then 
$$\Zc(\gg)=\spa\{X_1\},\ 
\Cc_2\gg=\spa\{X_1,X_2\},\ 
\Cc(\Cc_2\gg:\gg)=\spa\{X_1,X_2,X_3\},$$
hence we may choose $\cg=\spa\{X_2\}$, $\Vc=\spa\{X_4\}$, 
and for   
$$\begin{aligned}
\kg&=\spa\{X_1,X_2,X_4\} \quad\; \text{(Heisenberg algebra)}\\
\ag&=\spa\{X_1,X_3\} \quad\quad\quad \text{(abelian algebra)}
\end{aligned}$$ 
we have $\gg=\kg\redrtimes\ag$ of type~($A$). 
\end{example}

\begin{example}[Type ($A_{+}$)]\label{N5N3}
\normalfont
Let $\gg=\spa\{X_1,X_2,X_3,X_4\}$ with 
$$\begin{aligned}{}
[X_5,X_4]&=X_2,\ [X_5,X_2]=X_1, \\
[X_4,X_3]&=X_1. 
\end{aligned}$$
Then 
$$\Zc(\gg)=\spa\{X_1\},\ 
\Cc_2\gg=\Cc(\Cc_2\gg:\gg)=\spa\{X_1,X_2,X_3\},$$
hence we may choose $\cg=\spa\{X_2,X_3\}$, $\Vc=\spa\{X_4,X_5\}$, 
and $\gg$ is of type~($A_{+}$). 
\end{example}

\begin{example}[Type~($A$)]\label{N5N6}
\normalfont
Let $\gg=\spa\{X_1,X_2,X_3,X_4,X_5\}$ with 
$$\begin{aligned}{}
[X_5,X_4]&=X_3,\ [X_5,X_3]=X_2,\ [X_5,X_2]=X_1, \\
[X_4,X_3]&=X_1. 
\end{aligned}$$
Then 
$$\Zc(\gg)=\spa\{X_1\},\ 
\Cc_2\gg=\spa\{X_1,X_2\},\ 
\Cc(\Cc_2\gg:\gg)=\spa\{X_1,X_2,X_3,X_4\},$$
hence we may choose $\cg=\spa\{X_2\}$, $\Vc=\spa\{X_5\}$, 
and for   
$$\begin{aligned}
\kg&=\spa\{X_1,X_2,X_5\} \quad \text{(Heisenberg algebra)}\\
\ag&=\spa\{X_1,X_3,X_4\} \quad \text{(Heisenberg algebra)}
\end{aligned}$$ 
we have $\gg=\kg\redrtimes\ag$ of type~($A$) with flat generic coadjoint orbits. 
\end{example}

\begin{example}[Mixed type]\label{N6N1}
\normalfont
Let $\gg=\spa\{X_1,X_2,X_3,X_4,X_5,X_6\}$ with 
$$\begin{aligned}{}
[X_6,X_5]&=X_4,\ [X_6,X_4]=X_1,  \\
[X_3,X_2]&=X_1. 
\end{aligned}$$
Then 
$$\Zc(\gg)=\spa\{X_1\},\ 
\Cc_2\gg=\spa\{X_1,X_2,X_3,X_4\}.$$
Since $\Cc_2\gg$ fails to be abelian, more precisely $\Zc(\Cc_2\gg)=\spa\{X_1,X_4\}$, 
we choose its subalgebra 
$$\hg=\spa\{X_1,X_2,X_3\},$$ 
which is a Heisenberg algebra with $\Cc_2\gg=\Zc(\Cc_2\gg)+\hg$. 
Therefore, as in the proof of Theorem~\ref{struct}\eqref{struct_item1}, 
$\gg=\hg\redtimes\Cc(\hg:\gg)$, 
where 
$$\Cc(\hg:\gg)=\spa\{X_1,X_4,X_5,X_6\}. $$
Note that $\Cc(\hg:\gg)$ is isomorphic to the algebra of type~(A) from Example~\ref{N4N1}. 
\end{example}

\begin{example}[Type~($A$)]\label{N6N4}
\normalfont
Let $\gg=\spa\{X_1,X_2,X_3,X_4,X_5,X_6\}$ with 
$$\begin{aligned}{}
[X_6,X_5]&=X_3,\ [X_6,X_4]=X_2,  \\
[X_5,X_2]&=X_1, \\
[X_4,X_3]&=X_1. 
\end{aligned}$$
Then 
$$\Zc(\gg)=\spa\{X_1\},\ 
\Cc_2\gg=\spa\{X_1,X_2,X_3\},\  
\Cc(\Cc_2\gg:\gg)=\spa\{X_1,X_2,X_3,X_6\},$$
hence we may choose $\cg=\spa\{X_2,X_3\}$, $\Vc=\spa\{X_4,X_5\}$, 
and for   
$$\begin{aligned}
\kg&=\spa\{X_1,X_2,X_3,X_4,X_5\} \quad \text{(Heisenberg algebra)}\\
\ag&=\spa\{X_1,X_6\} \quad\quad\quad\quad\quad\quad\  \text{(abelian algebra)}
\end{aligned}$$ 
we have $\gg=\kg\redrtimes\ag$ of type~($A$). 
\end{example}

\begin{example}[Type~($A$)]\label{N6N5}
\normalfont
Let $\gg=\spa\{X_1,X_2,X_3,X_4,X_5,X_6\}$ with 
$$\begin{aligned}{}
[X_6,X_5]&=X_3,\ [X_6,X_4]=X_2,\ [X_6,X_3]=X_1,  \\
[X_4,X_2]&=X_1. 
\end{aligned}$$
Then 
$$\Zc(\gg)=\spa\{X_1\},\ 
\Cc_2\gg=\spa\{X_1,X_2,X_3,\},\  
\Cc(\Cc_2\gg:\gg)=\spa\{X_1,X_2,X_3,X_5\},$$
hence we may choose $\cg=\spa\{X_2,X_3\}$, $\Vc=\spa\{X_4,X_6\}$, 
and for   
$$\begin{aligned}
\kg&=\spa\{X_1,X_2,X_3,X_4,X_6\} \quad \text{(the algebra of type ($A_{+}$) from Example~\ref{N5N3})}\\
\ag&=\spa\{X_1,X_5\} \quad\quad\quad\quad\quad\quad\  \text{(abelian algebra)}
\end{aligned}$$ 
we have $\gg=\kg\redrtimes\ag$ of type~($A$). 
\end{example}

\begin{example}[Type~($A$)]\label{N6N6}
\normalfont
Let $\gg=\spa\{X_1,X_2,X_3,X_4,X_5,X_6\}$ with 
$$\begin{aligned}{}
[X_6,X_5]&=X_3,\ [X_6,X_4]=X_2,  \\
[X_5,X_3]&=X_1, \\
[X_4,X_2]&=X_1. 
\end{aligned}$$
Then 
$$\Zc(\gg)=\spa\{X_1\},\ 
\Cc_2\gg=\spa\{X_1,X_2,X_3\},\  
\Cc(\Cc_2\gg:\gg)=\spa\{X_1,X_2,X_3,X_6\},$$
hence we may choose $\cg=\spa\{X_2,X_3\}$, $\Vc=\spa\{X_4,X_5\}$, 
and for   
$$\begin{aligned}
\kg&=\spa\{X_1,X_2,X_3,X_4,X_5\} \quad \text{(Heisenberg algebra)}\\
\ag&=\spa\{X_1,X_6\} \quad\quad\quad\quad\quad\quad\  \text{(abelian algebra)}
\end{aligned}$$ 
we have $\gg=\kg\redrtimes\ag$ of type~($A$). 

Note that the above algebras $\kg$ and $\ag$ are the same as to the ones 
that occur in Example~\ref{N6N4}, and yet the present algebra $\gg$ is not isomorphic to the one 
from that example. 
This is due to the fact that mappings $\kg\times\ag\to\kg$ defined by the Lie bracket 
in the two examples are different from each other.  
\end{example}

\begin{example}\label{cohomology}
\normalfont 
For the sake of completeness, we recall here the method to construct 
3-step nilpotent Lie algebras with 1-dimensional centers and with flat coadjoint orbits, 
by starting from 2-step nilpotent Lie algebras with symplectic structures; 
see \cite{GKM04} for differential geometric implications of this and some related constructions. 

Let 
$\gg_0$ be a Lie algebra with a skew-symmetric bilinear form
$\omega\colon\gg_0\times\gg_0\to\KK$  satisfying 
$\omega([X,Y],Z)+\omega([Y,Z],X)+\omega([Z,X],Y)=0$ for all $X,Y,Z\in\gg$, 
that is, we have a scalar 2-cocycle $\omega\in Z^2(\gg_0,\KK)$, . 
Then $\gg:=\gg_0\dotplus_\omega\KK$ is the Lie algebra with the bracket 
$[(X_1,t_1),(X_2,t_2)]=([X_1,X_2],\omega(X_1,X_2))$ for all $X_1,X_2\in\gg_0$ and $t_1,t_2\in\KK$. 
If we define $\gg_0^{\perp_\omega}=\{X\in\gg_0\mid\omega(X,\gg_0)=0\}$, 
then it is easily seen that $\gg$ has 1-dimensional center if and only if 
$\Zc(\gg_0)\cap \gg_0^{\perp_\omega}=0$. 
In particular, this is the case if $\omega$ is a symplectic form, hence in addition to $\omega\in Z^2(\gg_0,\KK)$ 
it also satisfies $\gg_0^{\perp_\omega}=0$. 
We also note that if $\gg_0$ is an $n$-step nilpotent Lie algebra, then $\gg$ is $(n+1)$-step nilpotent Lie algebra. 
Moreover, if the center of $\gg$ is 1-dimensional 
then there exists an isomorphism of Lie algebras $\gg_0\simeq\gg/\Zc(\gg)$, 
hence the isomorphism class of $\gg$ is uniquely determined by the isomorphism class of $\gg_0$, 
in the sense that if $\gg_1/\Zc(\gg_1)\not\simeq\gg_2/\Zc(\gg_2)$, then $\gg_1\not\simeq\gg_2$.  

It follows by these remarks that one can construct nonisomorphic 3-step nilpotent Lie algebras with 1-dimensional center and flat generic coadjoint orbits by starting from nonisomorphic 2-step nilpotent Lie algebras endowed with symplectic structures. 
Examples of such kind of 2-step nilpotent Lie algebras over $\KK=\RR$ 
can be found in \cite{DT00}, \cite{GKM04}, and \cite{PT09}. 
See also \cite{Bu06} for examples of higher nilpotency order over $\KK=\CC$. 
\end{example}


\textbf{Acknowledgment.}  
We wish to thank the Referee for valuable suggestions that improved the presentation of this paper. 
This research has been partially supported by the Grant
of the Romanian National Authority for Scientific Research, CNCS-UEFISCDI,
project number PN-II-ID-PCE-2011-3-0131, 
and by Project MTM2013-42105-P, DGI-FEDER, of the MCYT, Spain.

\end{document}